\documentclass[12pt]{amsart}
\usepackage{amsmath,amssymb,xcolor}

\title{Varieties of planes on intersections of three quadrics}
\author{Brendan Hassett and Yuri Tschinkel}
\date{April 2019}

\theoremstyle{plain}

\theoremstyle{plain}
\newtheorem{prop}{Proposition}

\newtheorem{theo}[prop]{Theorem}

\theoremstyle{definition}

\newtheorem{ques}[prop]{Question}

\newtheorem{rema}[prop]{Remark}

\newcommand{\bC}{\mathbb C}

\newcommand{\bG}{\mathbb G}
\newcommand{\bP}{\mathbb P}
\newcommand{\bQ}{\mathbb Q}
\newcommand{\bZ}{\mathbb Z}

\newcommand{\cC}{\mathcal C}
\newcommand{\cE}{\mathcal E}
\newcommand{\cI}{\mathcal I}
\newcommand{\cO}{\mathcal O}

\newcommand{\cX}{\mathcal X}

\newcommand{\ra}{\rightarrow}

\newcommand{\Bl}{\operatorname{Bl}}
\newcommand{\Br}{\operatorname{Br}}
\newcommand{\Gr}{\operatorname{Gr}}

\newcommand{\Sym}{\operatorname{Sym}}
\newcommand{\Pic}{\operatorname{Pic}}
\newcommand{\prim}{\operatorname{prim}}

\begin{document}
\maketitle

\begin{abstract}
We study the geometry of spaces of planes on smooth complete intersections of three quadrics, with a view toward rationality questions.
\end{abstract}

\section{Introduction}

This note studies the geometry of smooth complete intersections of three quadrics $X\subset \bP^n$, with a view toward rationality questions over nonclosed fields $k$.  We review what is known over $\bC$:
\begin{itemize}
\item{$X$ is irrational for $n\le 6$ \cite{BeauvilleThesis};}
\item{$X$ may be either rational or irrational for $n=7$, and the rational ones are dense in moduli \cite{HPT17};}
\item{$X$ is always rational for $n\ge 8$ \cite[Cor.~5.1]{Tjurin}.}
\end{itemize}
The analysis in higher dimensions relies on the geometry of planes in $X$. 
Indeed, when $X$ contains a plane $P$ defined over $k$ then projection from $P$ gives a birational map
$$\pi_{P}: X \stackrel{\sim}{\dashrightarrow} \bP^{n-3}.$$
This leads us to study the variety of planes $F_2(X) \subset \Gr(3,n+1)$. When $n\ge 12$, the geometry
of these varieties gives a quick and uniform proof of rationality over finite fields and function fields 
of complex curves (see Theorem~\ref{theo:high}).  

We are therefore interested in the intermediate cases $n=8,9,10,11$, and especially in $n=8$ and $9$. 
For generic $X\subset \bP^8$, the variety $F_2(X)$ is finite of degree $1024$ (see Proposition~\ref{prop:finite});
we explore the geometry of the associated configurations of planes in $\bP^8$. We then focus most on the
case $n=9$. Here the variety $F_2(X)$ is a threefold of general type with complicated geometry -- we analyze its numerical invariants.

Our original motivation was to understand certain {\em singular} complete intersections of three quadrics in 
$\bP^9$ associated with universal torsors over degree 4 del Pezzo surfaces fibered in conics over $\bP^1$ \cite{CTSan}.
Conjectures of Colliot-Th\'el\`ene and Sansuc predict that such torsors are {\em rational} when they admit a point, over number fields.
Rationality of torsors has significant arithmetic applications, e.g., to proving the uniqueness of the Brauer--Manin obstruction
to the Hasse principle and weak approximation. It has geometric consequences as well, e.g., 
the construction of new examples of nonrational but stably rational threefolds over $\bC$. 
The geometry of {\em smooth} intersections, presented here, turned out to be quite rich and interesting on its own.    

Here is a road map of the paper: 
Section~\ref{sect:high} presents uniform proofs of rationality for high-dimensional cases. Section~\ref{sect:determinant} is
devoted to determinantal presentations of plane curves arising as degeneracy loci of the net of quadrics. 
In Sections~\ref{sect:schubert} and \ref{sect:koszul}, we turn to numerical invariants of the variety of planes. 
Our main results concern the computation of 
degrees and 
cohomology 
of $F_2(X)$, for $X$ a smooth intersection of three quadrics, in $\bP^8$ and $\bP^9$. 

\

\noindent
{\bf Acknowledgments:} 
The first author was partially supported by NSF grants 1551514 and 1701659, and the Simons Foundation; the second author was partially supported by NSF grant 1601912.

\section{Uniform rationality in high dimensions}
\label{sect:high}
\begin{theo}  \label{theo:high}
Let $k$ be a finite field or the function field of a complex curve. 
Suppose that $X\subset \bP^n, n\ge 12,$ is a smooth complete
intersection of three quadrics. Then $X$ is rational.
\end{theo}
\begin{proof}
The results of \cite{DebMan} show that for $X$ generic over $\bar{k}$,
the variety $F_2(X)\subset \Gr(3,n+1)$ is smooth and connected of the expected dimension $3n-24$. 
The adjunction formula implies that $F_2(X)$ has canonical class
$$K_{F_2(X)} = (12 - n-1) \sigma_1;$$
here $\sigma_1$ is the hyperplane class from the Pl\"ucker embedding -- see Section~\ref{sect:schubert}
for details. In particular, $F_2(X)$ is Fano for $n\ge 12$. 

Suppose $k$ is a function field. When $F_2(X)$ is smooth of the expected dimension, we have
$F_2(X)(k)\neq \emptyset$, by the Graber-Harris-Starr theorem \cite{GHS}. 
What if $F_2(X)$ is singular or fails to have the expected dimension?
However, we still have rational points, by an argument of Starr: Suppose that $\cX_{\circ} \ra B$ is a family 
of varieties over a complex projective curve $B$ corresponding to a morphism $B \ra \operatorname{Hilb}$
to the appropriate Hilbert scheme. There exists a one-parameter family of curves $\cC_t$, meeting
the locus over which $F_2$ is smooth of the expected dimension, such that $\cC_0$ contains $B$
as an irreducible component. (Take the $\cC_t$ to be complete intersection curves in $\operatorname{Hilb}$.)
For the induced families $\cX_t \ra \cC_t$, the fibrations $F_2(\cX_t)\ra \cC_t$ admit sections
for most $t$. This remains true as $t\ra 0$; restricting to $B$ gives a section of $F_2(\cX_{\circ})\ra B$.

A similar argument holds over finite fields, using Esnault's Theorem \cite{Esnault} and the specialization
version due to Fakhruddin and Rajan \cite{FR}.
\end{proof}

\section{Determinantal representations}
\label{sect:determinant}

\subsection{Recollection of invariants} \label{subsect:recall}
We recall formulas that may be obtained from Appendix I of
Hirzebruch's {\em Topological Methods}, specifically,
\cite[Th.~22.1.1]{hirzebruch}.

\begin{prop}
Suppose that $n=2m$. Then we have
$$\chi(X)=-4m(m-1)$$
and
$$h^{m-1,m-2}(X)=h^{m-2,m-1}(X)=\binom{2m}{2}-1=2m^2-m-1.$$
The other Hodge numbers of weight $2m-3$ vanish.
\end{prop}

\begin{prop}
Suppose that $n=2m-1$. Then we have
$$\chi(X)=4m(m-1)$$
and
$$h^{m-1,m-3}(X)=h^{m-3,m-1}(X)=\binom{m-1}{2}, 
h^{m-2,m-2}(X)=3m^2-3m+2.$$
\end{prop}
The odd case follows from \cite{ogrady86} and \cite{Laszlo}; the even case
from \cite{BeauvilleThesis}. For reference, we give a tabulation of nontrivial middle Hodge numbers
$$
\begin{array}{r|c}
n & \text{Hodge numbers}\\
\hline
4 & 5 \quad 5 \\
5 & 1 \quad 20 \quad 1 \\
6 & 14 \quad 14 \\
7 & 3 \quad 38 \quad 3 \\
8 & 27 \quad 27 \\
9 & 6 \quad 62 \quad 6 \\
10 & 44 \quad 44 
\end{array}
$$

\subsection{Interpretation of cohomology}
Let $X\subset \bP^n$ be a smooth complete intersection of three
quadrics containing a plane $P$. It follows that $n\ge 7$ and $X$ is of
special moduli when $n=7$, as the expected dimension of the 
Fano variety of planes is
$$\operatorname{expdim} F_2(X) = 3(n-2) - 3\cdot 6 = 3n-24.$$
This case is studied in depth in \cite{HPT17}.

Projection from $P$ induces a birational map 
$X \stackrel{\sim}{\dashrightarrow}\bP^{n-3}$. The center $Y$ of the inverse has 
dimension $n-5$ and admits a fibration in quadric hypersurfaces of relative
dimension $n-7$ over $\bP^2$.
This is instrumental in computing and interpreting the cohomology of $X$.

Suppose $n$ is odd. Then $Y$ is fibered in even-dimensional quadrics and its
cohomology is governed by the associated double cover 
$$
S \ra \bP^2,
$$
branched over the degeneracy locus $D$, and the associated Brauer class $\eta \in \Br(S)[2]$,
when $n\ge 9$. It is possible to interpret $F_2(X)$ via moduli spaces of vector bundles over
$S$  \cite{Bhosle}.

When $n$ is even, $Y$ is fibered in odd-dimensional quadrics, degenerate over a plane
curve $D$ of degree $n+1$. The cohomology is obtained via a Prym construction arising from 
the associated double cover of $D$. 

\subsection{Equations of the center}

In general, write
$$\cE:= \cO^3_{\bP^{n-3}} \oplus \cO_{\bP^{n-3}}(-1) \hookrightarrow 
\cO_{\bP^{n-3}}^{n+1}$$
so that $\bP(\cE)=\Bl_P(\bP^n)$ with bundle morphism 
$\varpi:\bP(\cE) \ra \bP^{n-3}$. 
Quadrics containing $P$ correspond to elements
of 
$$\Gamma(\cO_{\bP(\cE)}(1) \otimes \varpi^*\cO_{\bP^{n-3}}(1)) = 
\Gamma(\cE^*(1)),
$$
and complete intersections of three such forms correspond to linear
transformations
$$\cO_{\bP^{n-3}}^3 \ra \cE^*(1)\simeq \cO_{\bP^{n-3}}(1)^3 \oplus \cO_{\bP^{n-3}}(2).$$
We are interested in their degeneracy loci, whose equations are given 
by minors of            
\begin{equation} \label{eqn:4by3}
\left( \begin{matrix} L_{11} & L_{12} & L_{13} \\
                        L_{21} & L_{22} & L_{23} \\
                        L_{31} & L_{32} & L_{33} \\
                        Q_1 & Q_2 & Q_3 \end{matrix} \right) 
\end{equation}
where the $L_{ij}$ are linear and the $Q_i$ are quadratic; maximal
minors yield one cubic equation and three quartic equations.

Let $Y\subset \bP^{n-3}$ denote the subscheme given by the maximal minors,
which generically has codimension two. The $(n+1)$-dimensional
linear series
$\Gamma(\cI_Y(4))$
induces a birational map
$$\bP^{n-3} \dashrightarrow X \subset \bP^n$$
inverse to $\pi_P$.  

The determinantal equations (\ref{eqn:4by3}) for $Y$ yield a rational map
$$\phi: Y \dashrightarrow \bP^2,$$
assigning to each rank-two matrix its one-dimensional kernel.
This fails to be defined along the locus $Z \subset Y$ where the
matrix (\ref{eqn:4by3}) has rank one. See Section~\ref{subsect:generic}
for more information about its geometry.

We shall analyze the cohomology of $X$ using the
cohomology of natural resolutions of $Y$, expressed via $\phi$ as quadric
bundles over $\bP^2$.

\subsection{Linear algebra}
We start with some linear algebra on generic $3\times 3$
matrices
$$R_1:=\bP^2 \times \bP^2  \subset R_2 \subset \bP^8,$$
defined by the various minors of $L=(L_{ij})$. 
Over $R_1$, we may write 
$$L_{ij}=u_iv_j, \quad i,j=1,2,3.$$
The $2\times 2$ minors induce the `Cramer map'
$$\iota: \Bl_{R_1}(\bP^8) \ra \bP^8.$$
Let $E$ denote the exceptional divisor and $\widetilde{R_2}$
the proper transform of $R_2$.
The target space is also stratified by rank
$$R'_1 \subset R'_2 \subset \bP^8$$
with $\iota(\widetilde{R_2}) \subset R'_1$ and $\iota(E)=R'_2.$

Let $M=(M_{ij})$ denote the natural coordinates on the target $\bP^8$,
so the graph of $\iota$ is given by
$$L\cdot M = M\cdot L = \text{diagonal matrix}.$$ 
In other words, we have the equations
$$ \sum_j L_{ij}M_{jk} = \sum_{\ell} M_{i\ell} L_{\ell k}=0 \quad i\neq k$$
and 
$$ \sum_j L_{ij}M_{ji} = \sum_{\ell} M_{k\ell} L_{\ell k} \quad i=k.$$
Thus $\widetilde{R_2}$ is given by
$$\det(L_{ij})=0, \quad M_{ij}M_{k\ell}-M_{i\ell}M_{kj}=0, \quad
L\cdot M = M\cdot L =0.$$
Writing $M_{ij}=x_iy_j$ yields equations
$$
\sum_j L_{ij}x_j = \sum_{\ell} y_{\ell} L_{\ell k}=0,
$$
for all indices $i,k$.

We obtain two small resolutions of $R_2$ by keeping track of 
just the kernel or the image of $L$, respectively. For example, we
may consider
$$\widehat{R_2} \subset \bP^8_{L_{ij}} \times \bP^2_{x_j}$$
given by the equations
$$L \left( \begin{matrix} x_1 \\ x_2 \\ x_3 \end{matrix} \right)=0.$$
The exceptional locus $\widehat{E} \subset \widehat{R_2}$ is a
$\bP^1$-bundle over $R_1$ -- for each rank one matrix we extract the
one-dimensional subspaces of its kernel. 
It has codimension two in $\widehat{R_2}$.

\subsection{Generic behavior} \label{subsect:generic}
Assume that $n=11$ and the $L_{ij}$ are
linearly independent.

On the locus $R_1$ where the upper $3\times 3$ matrix
has rank at most one, we 
may write $L_{ij}=u_iv_j$.
The remaining equations involving the $Q_i$ take the form
$$u_1v_2Q_1-u_1v_1Q_2 = u_2v_1Q_1-u_2v_2Q_2 = \cdots = 0.$$
Writing $q_i=Q_i(u_1v_1,\ldots,u_3v_3)$, we obtain
$$v_2q_1 - v_1q_2 = v_3q_1 - v_1q_3 = v_2 q_3 - v_3 q_2 =0.$$
Thus there exists an
$$F \in k[u_1,u_2,u_3;v_1,v_2,v_3]_{(2,1)}$$
such that 
$$
q_1=v_1F, \quad q_2=v_2F, \quad q_3=v_3F.
$$
The indeterminacy locus for the kernel map
$\phi$ is given by a bidegree $(2,1)$ hypersurface
$$
Z \subset \bP^2 \times \bP^2.
$$  

Resolving the kernel map
$\phi: Y \dashrightarrow \bP^2$ yields a 
small resolution $\widehat{Y} \ra Y$ with center $Z$.
Write the kernel of (\ref{eqn:4by3}) as 
$[x_1,x_2,x_3]$
so that
$$[L_{12}L_{23}-L_{13}L_{23},-L_{11}L_{23}+L_{13}L_{23},L_{11}L_{22}-L_{12}L_{21}]\sim[x_1,x_2,x_3]$$
etc.
The rank two matrices with this kernel correspond to a codimension-three
linear subspace 
$$\Lambda_{[x_1,x_2,x_3]}\simeq \bP^5 \subset \bP^8.$$
The $3\times 3$ minors involving the $Q_i$ are all proportional to
$$Q_1x_1+Q_2x_2+Q_3x_3.$$
Hence 
$$\phi:\widehat{Y} \ra \bP^2$$
is a quadric bundle of relative dimension four.

\subsection{Degenerate cases in small dimension}

\

\noindent
{$\mathbf{n=5}$}: 
Suppose we have a complete intersection of three quadrics 
in $\bP^5$ containing a plane
$$X=P \cup_B U.$$
Here $\pi_P:U \ra \bP^2$ is birational, realizing
$U=\Bl_Y(\bP^2)$, where 
$$\phi:Y=\{y_1,\ldots,y_9\} \hookrightarrow \bP^2$$
is a generic collection of nine points, and the imbedding
$$
Y \hookrightarrow \bP^5
$$ 
is via quartics vanishing at those nine points. 

Note there is a canonical cubic curve
$W \supset Y$ where the determinant of linear forms
vanishes. For each divisor
$$\Sigma \equiv_W 4g-Y$$
there is a unique quartic form -- modulo the defining equation of $W$ --
cutting out $\Sigma \cup Y$. 

\

\noindent
{$\mathbf{n=6}$}: 
Now consider a complete intersection of three quadrics in $\bP^6$ containing
a plane
$$
P \subset X \subset \bP^6.
$$
The threefold $X$ is a nodal Fano threefold, with six singularities along 
$P$. Note that $X$ depends on $21$ parameters, codimension six in
the parameter space of all complete intersections.

In this case, we have
$$
Y \subset W \subset \bP^3,
$$
where $W$ is a smooth cubic surface, $Y$ is a smooth curve with
$$
\deg(Y)=9, \quad \operatorname{genus}(Y)=9,
$$
residual to a twisted cubic $\Sigma \subset W$ in the complete intersection
of $W$ and a quartic.  
Write
$$\Pic(W)=\left<L,E_1,E_2,E_3,E_4,E_5,E_6\right>, \quad L=[\Sigma], $$
where the $E_i$ are pairwise disjoint $(-1)$-classes with $L\cdot E_i=0$.
Then we have
$$[Y]=11L-4E_1-\ldots-4E_6$$
and the residual twisted cubic to $\Sigma$, 
with divisor class
$$
5L-2E_1-\ldots-2E_6,
$$ 
realizes $Y$ is a septic plane curve with
six nodes. The corresponding linear series induces $\phi:Y \ra \bP^2$. 

The intermediate Jacobian of a minimal resolution
$\widetilde{X} \ra X$ is isomorphic to the Jacobian of $Y$.

\

\noindent
{$\mathbf{n=7}$}: 
The simplest smooth case, obtained by
taking a codimension-four linear section of the generic case.   
For generic choices of the linear section, $Z=\emptyset$ and
$\phi:Y \ra \bP^2$ is a double cover branched over a special octic
plane curve.
Thus we have 
$$
Y \subset W:=\{\det(L_{ij})=0 \} \subset \bP^4,
$$
where $W$ is a cubic threefold with six ordinary double points \cite{HT10}.

The surface $Y$ has Picard group
$$\begin{array}{c|cc}
    & K & h \\
\hline
K  &  2 & 7 \\
h  &  7 & 9
\end{array}
$$
where $K$ is the canonical class -- pulled back from $\bP^2$ -- and
$h$ is associated with the embedding 
$$Y\hookrightarrow W \hookrightarrow \bP^4.$$
There is a second morphism $Y \rightarrow \bP^2$ associated with the
divisor $2h-K$ because
\begin{align*}
\chi(2h-K)& =\frac{(2h-K)^2-(2h-K)K}{2}+4 \\
                &= \frac{36-28+2-14+2}{2}+4=3.
\end{align*}
The residual intersection to $Y$ in the complete intersection of
$W$ and a quartic hypersurface is a cubic scroll $\Sigma$.
By \cite{HT10}, $Y$ admits two families of such scrolls,
each parametrized by $\bP^2$. An adjunction computation shows that
$$\Sigma \cap W \equiv_Y 2h-K.$$  

\begin{rema}
The lattice $\left<K,h\right>$ has discriminant $-31$. 
Consider the lattice
$$\left<g^2,P\right> \subset H^4(X,\bZ)$$
under the intersection pairing, where $g$ is the hyperplane class.
A Chern-class computation gives 
$P^2=4$ whence
$$\begin{array}{c|cc}
    & g^2 & P \\
\hline
g^2  &  8 & 1 \\
P    &  1 & 4
\end{array},
$$
which has discriminant $31$. 
The birational parametrization of $X$ induces an isomorphism of
Hodge structures
$$H^2(Y,\bZ)(-1) \supset \left< K,h\right>^{\perp} \simeq
\left< g^2,P\right>^{\perp} \subset H^4(X,\bZ).$$
\end{rema}

\

\noindent
{$\mathbf{n=8}$}: 
Here the $L_{ij}$ are linear forms on $\bP^5$ so the determinant cubic
is singular along an elliptic normal curve.  The center $Y\subset \bP^5$ of the mapping to $\bP^8$
is birational to a conic bundle over $\bP^2$. 
The locus $Z\subset Y$, where $Y$ meets the elliptic normal curve, 
consists of nine points.
The equations of $\widehat{Y} \subset \bP^5\times \bP^2$ may be written:
$$L_{i1}x_1+L_{i2}x_2+L_{i3}x_3=0, \quad
Q_1x_1 + Q_2x_2 + Q_3x_3=0.$$
The linear forms induce
$$0 \ra K \ra \cO_{\bP^2}^{\oplus 6} \ra \cO_{\bP^2}(1)^{\oplus 3} \ra 0 $$
and the quadratic form induces a symmetric map
$$K \ra K^*(1)$$
with degeneracy $C\subset \bP^2$ of degree $9$.  

We count parameters: the linear terms depend on $45$ parameters, the quadratic
on an additional $44$ parameters, and coordinate changes account for $43$ parameters,
leaving a total of $46$ parameters. We know that $C$ comes with a distinguished
double cover $\widetilde{C}\ra C$ but even after fixing this we obtain a total of $1024$ determinantal representations (see Proposition~\ref{prop:finite}).

\

\noindent
{$\mathbf{n=9}$}: 
Here the $L_{ij}$ are linear forms on 
$\bP^6$. This allows us to use the normal form
$$\left( \begin{matrix} A_1 & A_2 & A_3 \\
			A_5 & A_1 & A_4 \\
			A_6 & A_7 & A_1  \end{matrix} \right)
\left( \begin{matrix} x_1 \\ x_2 \\ x_3 \end{matrix} \right)=0.$$
The rank-one locus is a degree-six del Pezzo surface.
The locus $Z\subset Y$ is a curve of genus three and degree nine.
The equations of $\widehat{Y} \subset \bP^6\times \bP^2$ may be written:
\begin{align*}
A_1x_1+A_2x_2+A_3x_3=A_5x_1+A_1x_2+A_4x_3=A_6x_1+A_7x_2+A_1x_3 =&0 \\
Q_1x_1 + Q_2x_2 + Q_3x_3=&0.
\end{align*}
The linear forms induce
$$0 \ra K \ra \cO_{\bP^2}^{\oplus 7} \ra \cO_{\bP^2}(1)^{\oplus 3} \ra 0 $$
and the quadratic form induces a symmetric map
$$K \ra K^*(1)$$
with degeneracy $C\subset \bP^2$ of degree $10$.  

Fixing $C$ and the double cover $\widetilde{C}\ra C$, the various
determinant representations as above are parametrized by the
planes in $X$, i.e., by the threefold $F_2(X)$.

\section{Schubert calculus of the variety of planes}
\label{sect:schubert}

Now assume that the variety of planes
$$F_2(X) \subset \Gr(3,n+1)$$
is smooth of the expected dimension $3n-24$. 

Consider the canonical exact sequence over $\Gr(3,n+1)$
$$0 \ra S \ra V\otimes \cO \ra Q \ra 0.$$
The defining equations
$$
X = \{F_1=F_2=F_3=0 \}, \quad F_1,F_2,F_3 \in \Sym^2(V^*),
$$
induce sections $f_1,f_2,f_3 \in \Sym^2(S^*)$ so that
$$F_2(X) = \{ f_1=f_2=f_3=0\}.$$

We compute Chern classes. The Chern classes of $S$ are the Schubert
cycles
$$c(S) = 1 - \sigma_{1} + \sigma_{1,1} - \sigma_{1,1,1},
$$
yielding
$$c(S^*) = 1 + \sigma_{1} + \sigma_{1,1} + \sigma_{1,1,1}.$$
A computation in symmetric functions yields
\begin{eqnarray*}
c_1(\Sym^2(S^*)) &=& 4\sigma_{1} \\
c_2(\Sym^2(S^*)) &=& 5\sigma^2_{1}+5\sigma_{1,1}=5\sigma_{2,0}+10\sigma_{1,1} \\
c_3(\Sym^2(S^*)) &=& 2\sigma_{1}^3 + 11 \sigma_{1}\sigma_{1,1}+7\sigma_{1,1,1} \\
		 &=& 2\sigma_3 + 15\sigma_{2,1} + 20\sigma_{1,1,1} \\
c_4(\Sym^2(S^*)) &=& 6\sigma^2_{1}\sigma_{1,1} + 14 \sigma_{1}\sigma_{1,1,1}
+ 4 \sigma_{11}^2 \\
		 &=& 6\sigma_{3,1} + 10\sigma_{2,2} + 30\sigma_{2,1,1} \\
c_5(\Sym^2(S^*)) &=& 8\sigma^2_{1}\sigma_{1,1,1}+ 4\sigma_{1}\sigma_{1,1}^2
+4\sigma_{1,1}\sigma_{1,1,1} \\
		 &=& 4\sigma_{3,2} + 12 \sigma_{3,1,1} + 20\sigma_{2,2,1} \\
c_6(\Sym^2(S^*)) &=& 8 \sigma_{1}\sigma_{1,1}\sigma_{1,1,1}-8\sigma_{1,1,1}^2
			= 8 \sigma_{3,2,1}.
\end{eqnarray*}
Computing via the `SchurRings' package of \texttt{Macaulay2}, we obtain
{\small
$$[F_2(X)]=512(\sigma_{9,6,3}+2 \sigma_{9,5,4} + 2\sigma_{8,7,3}+6\sigma_{8,6,4}
+4\sigma_{8,5,5} + 4\sigma_{7,7,4}+8\sigma_{7,6,5}+2\sigma_{6,6,6}).$$
}
The paper \cite{DebMan} explores these formulas more systematically;
our formula extends the tabulation on \cite[p.~563]{DebMan}.
See also \cite{Jiang} for more information on Noether-Lefschetz
questions for Fano schemes.

\begin{prop}
\label{prop:finite}
When $n=8$ and $X$ generic, the variety  $F_2(X) \subset \Gr(3,9)$ has
dimension zero and 
$$[F_2(X)]=1024[\operatorname{point}].$$
\end{prop}

\begin{ques}
Describe the Galois action arising in this case.
To what extent is it governed by the Galois representation
on the intermediate Jacobian?
\end{ques}
Frank Sottile and his collaborators (see, e.g., \cite{HRS}) are developing computational 
approaches to Galois groups of enumerative questions. Computations
with Taylor Brysiewicz indicate the group in this case might be smaller
than the symmetric group $\mathfrak{S}_{1024}$ but a full computation
appears difficult with existing techniques and computational resources.
\begin{rema}
We propose to construct these examples synthetically: Fix seven generic planes
$$P_1,\ldots,P_7 \subset \bP^8 $$
which depend up to projectivity on $18\cdot 7 - 80 = 46$ parameters. 
Write 
$$\Pi=P_1 \sqcup P_2 \sqcup \cdots \sqcup P_7$$
and note that
$$h^0(\cI_{\Pi}(2))=h^0(\cO_{\bP^8}(2)) - 7\cdot 6 = 3.$$
Thus there is a {\em unique} complete intersection of three quadrics
$X\supset \Pi$ which we expect is smooth. How do we construct the 
other $1017$ planes of $F_2(X)$?
\end{rema}

We focus next on the case $n=9$ where $F_2(X) \subset \Gr(3,10)$
is a threefold with class
$$[F_2(X)]=512 (4\sigma_{7,7,4}+8\sigma_{7,6,5}+2\sigma_{6,6,6}).$$
We have
$$\sigma_1^3 = \sigma_3 + 2\sigma_{2,1}+\sigma_{1,1,1}$$
hence we have
$$\deg(F_2(X))=11264 = 11\cdot 2^{10}.$$
The adjunction formula implies
$$K_{F_2(X)}=2\sigma_1$$
whence
$$K_{F_2(X)}^3=11 \cdot 2^{13}.$$

\begin{prop} \cite[Th.~3.4]{DebMan} \label{prop:DebMan}
When $F_2(X)$ is smooth of the expected dimension we have
$$h^{0,1}(F_2(X))=h^{1,0}(F_2(X))=0.$$
\end{prop}

Consider the rank $18$ vector bundle
$$\cE = \Sym^2(S^*)^{\oplus 3}$$
so there exists a section $s\in \Gamma(\Gr(3,10),\cE)$
with $F_2(X)=\{s=0\}$. 
We have
$$\begin{array}{rcl}
c_1(\cE) &=& 3c_1(\Sym^2(S^*))=12\sigma_1 \\
c_2(\cE) &=& 3c_1(\Sym^2(S^*))^2+ 3c_2(\Sym^2(S^*))= 63 \sigma_2+78\sigma_{1,1}\\
c_3(\cE) &=& 190\sigma_3 + 533 \sigma_{2,1} + 364 \sigma_{1,1,1}. 
\end{array}
$$

The exact sequence
$$ 0 \ra T_{F_2(X)} \ra T_{\Gr(3,10)}|F_2(X) \ra N_{F_2(X)/\Gr(3,10)} \ra 0$$
and the interpretation of the normal bundle as $\cE|F_2(X)$
allows us to compute all the Chern classes of $T_{F_2(X)}$.

Indeed, $T_{\Gr(3,10)}=\operatorname{Hom}(S,Q)$ where 
$$c_1(Q)=\sigma_1, \quad c_2(Q)=\sigma_2, \quad c_3(Q)=\sigma_3.$$
$$
\begin{array}{rcl}
c_1(T_{\Gr(3,10)}) &=& r(S) c_1(Q) - r(Q) c_1(S) = 10\sigma_1 \\
c_2(T_{\Gr(3,10)}) &=& \binom{r(S)}{2}c_1(Q)^2 + r(S) c_2(Q) +  \binom{r(Q)}{2}c_1(S)^2 \\
& & 
		      \quad\quad\,   + r(Q) c_2(S) 
			- (r(S)r(Q)-1)c_1(Q)c_1(S) \\
		  &=& 3 \sigma_1^2 + 3 \sigma_2 + 21 \sigma_1^2 +
			7 \sigma_{1,1} -20 \sigma_1^2 \\
		  &=& 47 \sigma_2 + 51 \sigma_{1,1} \\
c_3(T_{\Gr(3,10)}) &=& 140 \sigma_3 + 310\sigma_{2,1} + 180\sigma_{1,1,1}.
\end{array}
$$
The last is computed using the identity
$$
\operatorname{ch}(T_{\Gr(3,10)})=\operatorname{ch}(S^*)\operatorname{ch}(Q),
$$
which yields
$$\operatorname{ch}_2(T_{\Gr(3,10)}) = 3\sigma_{2}-\sigma_{11}$$
and
$$\operatorname{ch}_3(T_{\Gr(3,10)}) = \frac{5}{3}(\sigma_{3}-\sigma_{2,1}+
					\sigma_{1,1,1}).$$
Note that
$$\operatorname{ch}_3(V)= \frac{1}{6}(c_1(V)^3-3c_1(V)c_2(V)+3c_3(V)).$$

We compute the Chern classes, identifying bundles with their
restrictions to $F_2(X)$:
As we have seen
$$c_1(T_{F_2(X)})= c_1(T_{\Gr(3,10)}) - c_1(\cE) = -2\sigma_1.$$
The second Chern class is:
$$
\begin{array}{rcl}
c_2(T_{F_2(X)}) &=& c_2(T_{\Gr(3,10)})-c_2(\cE)
			-c_1(T_{F_2(X)})c_1(\cE) \\
		&=& 8\sigma_2 - 3\sigma_{1,1}
\end{array}
$$
And finally
$$
\begin{array}{rcl}
c_3(T_{F_2(X)}) &=& c_3(T_{\Gr(3,10)})-c_2(T_{F_2(X)})c_1(\cE)
			-c_1(T_{F_2(X)})c_2(\cE) - c_3(\cE) \\
		&=& -20\sigma_3  - \sigma_{2,1}  + 8\sigma_{1,1,1}
\end{array}
$$

A computation with
the `SchurRings' package of \texttt{Macaulay2} gives
$$\chi(\cO_{F_2(X)})=c_1(T_{F_2(X)})c_2(T_{F_2(X)})/24 = -2816=-2^7\cdot 11$$
and
$$\chi(F_2(X))=c_3(T_{F_2(X)}) = -36,864 = -2^{12}\cdot 3^2 .$$

The Hirzebruch-Riemann-Roch formula gives
$$\chi(\cO_{F_2(X)}(1)) = 0,$$
which is to be expected, as Serre duality gives
$$h^i(\cO_{F_2(X)}(1))=h^{3-i}(\omega_{F_2(X)}(-1))=h^{3-i}(\cO_{F_2(X)}(1)).$$
Similarly, we obtain $\chi(\cO_{F_2(X)}(2))=2816$. 
We also get
$$\chi(\cO_{F_2(X)}(3))=h^0(\cO_{F_2(X)}(3)) = 16,896 = 2^9 \cdot 3 \cdot 11,  $$
where the first equality is Kodaira vanishing.
Thus we have enough data to extract the Hilbert polynomial of $F_2(X)$
$$\chi(\cO_{F_2(X)}(m))=\frac{2^7\cdot 11}{3}(m-1) ( 5 (m-1)^2 -2).$$
Hirzebruch-Riemann-Roch allows us to compute $\chi(\Omega^1_{F_2(X)})$
as well
$$\chi(\Omega^1_{F_2(X)}) =15,616= 2^{8} \cdot 61.$$

\section{Koszul computations}
\label{sect:koszul}

\subsection{General set-up}
Let $\cE$ denote a vector bundle and $s\in \Gamma(\cE)$ a section
such that the degeneracy locus $Z=\{s=0\}$ has codimension equal
to the rank $R$ of $\cE$. Then we have a resolution
$$0 \ra \bigwedge^{R} \cE^* \ra \bigwedge^{R-1} \cE^* \ra
\cdots \ra \cE^* \ra \cO  \ra \cO_Z \ra 0$$
and the Koszul complex
$$0 \ra \bigwedge^{R} \cE^* \ra \bigwedge^{R-1} \cE^* \ra
\cdots \ra \cE^* \ra \cO  \ra 0.$$
The arrow
$$\iota_s: \bigwedge^n \cE^* \ra \bigwedge^{n-1} \cE^*$$
is contraction by $s$. 
The hypercohomology spectral sequence gives
$$E_1^{p,q}= H^{q}(\bigwedge^{-p} \cE^*) \Rightarrow H^{p+q}(\cO_Z),$$
where $p=-R,\ldots,0$.

Following \cite{Manivel}, we may use this to compute the cohomology of 
degeneracy loci over Grassmannians $\bG(n,d-1)=\Gr(n+1,V)$. 
Given a $d$-dimensional vector space $V$ and
integers $\lambda = (\lambda_1,\ldots,\lambda_k)$
with $\lambda_1 \ge \lambda_2  \ge \cdots  \ge \lambda_k$
and $\lambda_k=0$ for $k>d$, let 
$\Gamma^{\lambda} V$ denote the associated Schur functor.
We observe the convention $\Gamma^{\lambda}V=0$ for sequences
of integers $\lambda$ failing the decreasing or vanishing conditions.

We keep track of the decompositions using the dictionary between 
Schubert calculus and tensor products
$$\Gamma^{\lambda} V \otimes \Gamma^{\mu} V = 
\sum_{\nu} c_{\lambda, \mu}^{\nu} \Gamma^{\nu} V,
$$ 
where the multiplicies are defined by
$$\sigma_{\lambda}\sigma_{\mu}= \sum_{\nu} c_{\lambda,\mu}^{\nu} \sigma_{\nu}.$$

\subsection{Planes in intersections of three quadrics in $\bP^9$}
Consider the bundle
$$\cE = \Sym^2(S^*)^{\oplus 3}$$
and $s\in \Gamma(\cE)$ with $F_2(X)=\{s=0\}$. 
We have the associated Koszul resolution
$$0 \ra \bigwedge^{18} \cE^* \ra \cdots \ra \bigwedge^r \cE^* \ra
\cdots \ra \cE^* \ra \cO  \ra \cO_{F_2(X)} \ra 0.$$
The terms of the Koszul complex decompose into direct sums of products
$$\bigwedge^{e_1} \Sym^2(S) \otimes \bigwedge^{e_2} \Sym^2(S) \otimes
\bigwedge^{e_3} \Sym^2(S), \quad e_1+e_2+e_3=r.$$
The hypercohomology spectral sequence gives
$$E_1^{p,q}= H^{q}(\bigwedge^{-p} \cE^*) \Rightarrow H^{p+q}(\cO_{F_2(X)}),$$
where $p=-18,\ldots,0$.

Recall the Weyl character formula for $\operatorname{GL}(V)$
representations:
$$\dim \Gamma^{\lambda}(V) = \prod_{i<j} \frac{\lambda_i-\lambda_j+j-i}{j-i}.$$
Assume $S$ has rank three then
$$\operatorname{rank} \Gamma^{(a_1,a_2,a_3)}S = (a_1-a_2+1)(a_1-a_3+2)(a_2-a_3+1)/2.$$

We observe the dictionary
$$\begin{array}{rcl}
\Sym^2(S) & \sim & \sigma_2 \\
\bigwedge^2 \Sym^2(S) & \sim & \sigma_{3,1} \\
\bigwedge^3 \Sym^2(S) & \sim & \sigma_{3,3} + \sigma_{4,1,1} \\
\bigwedge^4 \Sym^2(S) & \sim & \sigma_{4,3,1} \\
\bigwedge^5 \Sym^2(S) & \sim & \sigma_{4,4,2} \\
\bigwedge^6 \Sym^2(S) & \sim & \sigma_{4,4,4}
\end{array}
$$
whence
\begin{align*}
\Sym^2(S) &\simeq \Gamma^{(2)}(S) \\
\bigwedge^2 \Sym^2(S) & \simeq \Gamma^{(3,1)}(S) \\
\bigwedge^3 \Sym^2(S) & \simeq \Gamma^{(3,3)}(S) \oplus \Gamma^{(4,1,1)}(S) \\
\bigwedge^4 \Sym^2(S) & \simeq \Gamma^{(4,3,1)}(S) \\
\bigwedge^5 \Sym^2(S) & \simeq \Gamma^{(4,4,2)}(S) \\
\bigwedge^6 \Sym^2(S) & \simeq \Gamma^{(4,4,4)}(S)
\end{align*}

Note also \cite[Cor.~1.3]{Jiang} that 
$$H^j(\Gr(3,10),\Gamma^{\lambda} S) = 0 $$
whenever $j$ is not divisible by $7$ and for all but
one value of $j$.  

We tabulate the representations $\Gamma^a S, a=(a_1,a_2,a_3)$
appearing in each $\bigwedge^r \cE^*$:
{\small
$$
\begin{array}{r|l}
r & \text{decomposition} \\
\hline
0 & {\color{red} \Gamma^{(0,0,0)}} \\
1 & 3\Gamma^{(2)} \\
2 & 3\Gamma^{(4)} + 6\Gamma^{(3,1)} + 3\Gamma^{(2,2)} \\
3 & \Gamma^{(6)}+ 8\Gamma^{(5,1)} + 9\Gamma^{(4,2)} + 10\Gamma^{(4,1,1)} 
  +10\Gamma^{(3,3)}+8\Gamma^{(3,2,1)}+\Gamma^{(2,2,2)}\\
4 & 3\Gamma^{(7,1)}+9\Gamma^{(6,2)}+15\Gamma^{(6,1,1)}+18\Gamma^{(5,3)}
  +24\Gamma^{(5,2,1)}+6\Gamma^{(4,4)} \\
  & +33\Gamma^{(4,3,1)}+9\Gamma^{(4,2,2)}
  +15\Gamma^{(3,3,2)} \\
5 & 3\Gamma^{(8,2)}+{\color{red} 6\Gamma^{(8,1,1)}}+9\Gamma^{(7,3)}+24\Gamma^{(7,2,1)}+18\Gamma^{(6,4)}+54\Gamma^{(6,3,1)} \\
  & +21\Gamma^{(6,2,2)}+6\Gamma^{(5,5)}+48\Gamma^{(5,4,1)}+54\Gamma^{(5,3,2)}+39\Gamma^{(4,4,2)}+30\Gamma^{(4,3,3)}\\

6& \Gamma^{(9,3)}+8\Gamma^{(9,2,1)}+8\Gamma^{(8,4)}+27\Gamma^{(8,3,1)}+19\Gamma^{(8,2,2)}+9\Gamma^{(7,5)}\\
 & + 64\Gamma^{(7,4,1)}+62\Gamma^{(7,3,2)}+10\Gamma^{(6,6)}+53\Gamma^{(6,5,1)}+117\Gamma^{(6,4,2)}+
	56\Gamma^{(6,3,3)} \\
  &+ 46\Gamma^{(5,5,2)}+88\Gamma^{(5,4,3)}+38\Gamma^{(4,4,4)}\\

7& 3\Gamma^{(10,3,1)}+6\Gamma^{(10,2,2)}+3\Gamma^{(9,5)}+24\Gamma^{(9,4,1)}+27\Gamma^{(9,3,2)}+6\Gamma^{(8,6)} \\
 & + 42\Gamma^{(8,5,1)} + 93\Gamma^{(8,4,2)} + 36\Gamma^{(8,3,3)} + 3\Gamma^{(7,7)}+48\Gamma^{(7,6,1)}+ 132\Gamma^{(7,5,2)}\\
 & +144\Gamma^{(7,4,3)}+66\Gamma^{(6,6,2)}+138\Gamma^{(6,5,3)}+114\Gamma^{(6,4,4)}+60\Gamma^{(5,5,4)} \\

8 & 3\Gamma^{(11,3,2)}+9\Gamma^{(10,5,1)}+24\Gamma^{(10,4,2)}+9\Gamma^{(10,3,3)}+3\Gamma^{(9,7)}+24\Gamma^{(9,6,1)} \\
  & +75\Gamma^{(9,5,2)}+72\Gamma^{(9,4,3)}+24\Gamma^{(8,7,1)}+102\Gamma^{(8,6,2)}+168\Gamma^{(8,5,3)}+99\Gamma^{(8,4,4)}\\
  & +69\Gamma^{(7,7,2)}+168\Gamma^{(7,6,3)}+213\Gamma^{(7,5,4)}+96\Gamma^{(6,6,4)}+75\Gamma^{(6,5,5)} \\
9 & \Gamma^{(12,3,3)}+9\Gamma^{(11,5,2)}+8\Gamma^{(11,4,3)}+9\Gamma^{(10,7,1)}+36\Gamma^{(10,6,2)} + 63\Gamma^{(10,5,3)} \\
  & +28\Gamma^{(10,4,4)}+{\color{red} \Gamma^{(9,9)}}+8\Gamma^{(9,8,1)}+63\Gamma^{(9,7,2)}+128\Gamma^{(9,6,3)}+142 \Gamma^{(9,5,4)} \\
  & +28\Gamma^{(8,8,2)}+142\Gamma^{(8,7,3)}+216\Gamma^{(8,6,4)}+146\Gamma^{(8,5,5)} + 146\Gamma^{(7,7,4)}\\
  & +160\Gamma^{(7,6,5)}+20\Gamma^{(6,6,6)}
\end{array}
$$
}
The terms in {\color{red} red} contribute to the cohomology \cite[Th.~1.2]{Jiang}. 
The representations for $r>9$ 
may be read off via duality
$$\bigwedge^{r} \cE^* = \det(\cE^*) \otimes \bigwedge^{18-r} \cE, \quad
\Gamma^{(a_1,a_2,a_3)} \times \Gamma^{(12-a_3,12-a_2,12-a_1)} \ra \Gamma^{(12,12,12)}.$$

{\small
$$
\begin{array}{r|l}
r & \text{decomposition} \\
\hline
18 & {\color{red} \Gamma^{(12,12,12)}} \\
17 & {\color{red} 3\Gamma^{(12,12,10)}} \\
16 & 3\Gamma^{(12,12,8)} + 6\Gamma^{(12,11,9)} + {\color{red} 3\Gamma^{(12,10,10)}} \\
15 & \Gamma^{(12,12,6)}+ 8\Gamma^{(12,11,7)} + 9\Gamma^{(12,10,8)} + 10\Gamma^{(11,11,8)} 
  +10\Gamma^{(12,9,9)}+8\Gamma^{(11,10,9)}\\
   & +{\color{red} \Gamma^{(10,10,10)}}\\

14 & 3\Gamma^{(12,11,5)}+9\Gamma^{(12,10,6)}+15\Gamma^{(11,11,6)}+18\Gamma^{(12,9,7)}
  +24\Gamma^{(11,10,7)}+6\Gamma^{(12,8,8)} \\
  & +33\Gamma^{(11,9,8)}+9\Gamma^{(10,10,8)}
  +15\Gamma^{(10,9,9)} \\

13 & 3\Gamma^{(12,10,4)}+6\Gamma^{(11,11,4)}+9\Gamma^{(12,9,5)}+24\Gamma^{(11,10,5)}+18\Gamma^{(12,8,6)}+54\Gamma^{(11,9,6)} \\
  & +21\Gamma^{(10,10,6)}+6\Gamma^{(12,7,7)}+48\Gamma^{(11,8,7)}+54\Gamma^{(10,9,7)}+39\Gamma^{(10,8,8)}+30\Gamma^{(9,9,8)}\\

12& \Gamma^{(12,9,3)}+8\Gamma^{(11,10,3)}+8\Gamma^{(12,8,4)}+27\Gamma^{(11,9,4)}+19\Gamma^{(10,10,4)}+9\Gamma^{(12,7,5)}\\
 & + 64\Gamma^{(11,8,5)}+62\Gamma^{(10,9,5)}+10\Gamma^{(12,6,6)}+53\Gamma^{(11,7,6)}+117\Gamma^{(10,8,6)}\\
  &+56\Gamma^{(9,9,6)} 
  + 46\Gamma^{(10,7,7)}+88\Gamma^{(9,8,7)}+38\Gamma^{(8,8,8)}\\

11&{\color{red} 3\Gamma^{(11,9,2)}}+{\color{red} 6\Gamma^{(10,10,2)}}+3\Gamma^{(12,7,3)}+24\Gamma^{(11,8,3)}+27\Gamma^{(10,9,3)}+6\Gamma^{(12,6,4)} \\
 & + 42\Gamma^{(11,7,4)} + 93\Gamma^{(10,8,4)} + 36\Gamma^{(9,9,4)} + 3\Gamma^{(12,5,5)}+48\Gamma^{(11,6,5)}+ 132\Gamma^{(10,7,5)}\\
 & +144\Gamma^{(9,8,5)}+66\Gamma^{(10,6,6)}+138\Gamma^{(9,7,6)}+114\Gamma^{(8,8,6)}+60\Gamma^{(8,7,7)} \\

10 & {\color{red} 3\Gamma^{(10,9,1)}}+9\Gamma^{(11,7,2)}+24\Gamma^{(10,8,2)}+{\color{red} 9\Gamma^{(9,9,2)}}+3\Gamma^{(12,5,3)}+24\Gamma^{(11,6,3)} \\
  & +75\Gamma^{(10,7,3)}+72\Gamma^{(9,8,3)}+24\Gamma^{(11,5,4)}+102\Gamma^{(10,6,4)}+168\Gamma^{(9,7,4)}\\
   &+99\Gamma^{(8,8,4)}
 +69\Gamma^{(10,5,5)}+168\Gamma^{(9,6,5)}+213\Gamma^{(8,7,5)}+96\Gamma^{(8,6,6)}+75\Gamma^{(7,7,6)} 
\end{array}
$$
}

And then we record those contributing cohomology in degree $j$,
using  \cite[Th.~1.2]{Jiang}. Note that the sequence
$$(-1,-2,-3,-4,-5,-6,-7,a_1-8,a_2-9,a_3-10)$$
must have no repeating integers as these yield `singular'
weights. Thus we must have $a_1\ge 8$; if $a_1=8$ then $a_2=1,0$, etc.
$$\begin{array}{r|l}
j & \text{weights $a$ contributing } \\
\hline \hline
0 & (0,0,0) \\
\hline
7  & \text{no contributions in } \bigwedge^4 \cE^* \\
   & (8,1,1) \\
   & \text{no contributions in } \bigwedge^6 \cE^* \\
\hline
14 & (9,9,0) \\
   & (10,9,1), (9,9,2) \\
   & (10,10,2), (11,9,2) \\
   &  \text{no contributions in } \bigwedge^{12} \cE^* \\
\hline
21 & (10,10,10)  \\
   & (12,10,10) \\
  & (12,12,10)  \\
 & (12,12,12)
\end{array}$$

The degree $7$ cohomology contributes through
$$H^7(\bigwedge^5 \cE^*)=H^7((\Gamma^{(8,1,1)}S)^{\oplus 6})
\simeq (\Gamma^{(1,1,1,1,1,1,1,1,1,1)}V)^{\oplus 6}
\simeq (\det(V))^{\oplus 6} \simeq \bC^6.$$
Here $V$ is the standard $10$-dimensional representation.

\begin{prop} \label{prop:getsix}
We have
$$H^0(\cO_{F_2(X)})=\bC, \quad H^1(\cO_{F_2(X)})=0.$$
For higher degrees, we have
$$\bC^6 \simeq
E_1^{-5,7}=E_2^{-5,7}=\cdots = E_{\infty}^{-5,7}=H^2(\cO_{F_2(X)}).$$
Consequently, we deduce that 
$$h^3(\cO_{F_2(X)})=2816 + 1 + 6 - 0 = 2823.$$
\end{prop}
\begin{proof} The only term in degree zero is $E_1^{(0,0)}$ and there are no terms in degree one, which gives the degree $0$ and $1$ cohomology; see also Prop.~\ref{prop:DebMan} and \cite{DebMan}. The spectral sequence is supported at the following values of $(p,q)$:
\begin{align*}
(0,0), (-5,7), (-9,14),& (-10,14),(-11,14), \\
 (-15,21), (-16,21),&(-17,21), (-18,21).
 \end{align*}
All the values after $(-5,7)$ have degree $p+q \ge 3$ and $p\le -5$, thus 
do not receive arrows from $E_r^{-5,7}$. And clearly there are no maps from 
the degree $(0,0)$ term, abutting to $\Gamma(\cO_{F_2(X)})$. This yields
the equalities asserted above.
\end{proof}

The degree $14$ cohomology contributes through
\begin{eqnarray*}
H^{14}(\bigwedge^{11} \cE^*) &=& H^{14}((\Gamma^{(10,10,2)}S)^{\oplus 6})\oplus
					(\Gamma^{(11,9,2)}S)^{\oplus 3})) \\
			     &=& (\Gamma^{(3,3,2,2,2,2,2,2,2,2)}V)^{\oplus 6}
			\oplus (\Gamma^{(4,2,2,2,2,2,2,2,2)}V)^{\oplus 3} \\
		             &\simeq & (\bC^{45})^{6} \oplus (\bC^{55})^{3} 
					\simeq \bC^{435} \\
H^{14}(\bigwedge^{10} \cE^*) &=& H^{14}((\Gamma^{(9,9,2)}S)^{\oplus 9})\oplus
					(\Gamma^{(10,9,1)}S)^{\oplus 3})) \\
			   &=& (\Gamma^{(2,2,2,2,2,2,2,2,2,2)}V)^{\oplus 9}
			\oplus (\Gamma^{(3,2,2,2,2,2,2,2,2,1)}V)^{\oplus 3} \\
			&\simeq & (\bC)^9 \oplus (\bC^{99})^3 \simeq  \bC^{306} \\
H^{14}(\bigwedge^9 \cE^*) &=&  H^{14}(\Gamma^{(9,9)}S)  \\
 			  &=& \Gamma^{(2,2,2,2,2,2,2,2,2,0)}V
					\simeq \bC^{55} 
\end{eqnarray*}

The degree $21$ cohomology contributes through:
\begin{eqnarray*}
H^{21}(\bigwedge^{18} \cE^*) &=& H^{21}(\Gamma^{12,12,12}S) \\
		&=&\Gamma^{(5,5,5,3,3,3,3,3,3,3)}V\simeq \bC^{4950} \\
H^{21}(\bigwedge^{17} \cE^*) &=& H^{21}((\Gamma^{12,12,10}S)^{\oplus 3}) \\
		&=&(\Gamma^{(5,5,3,3,3,3,3,3,3,3)}V)^{\oplus 3}
			\simeq (\bC^{825})^3 \simeq \bC^{2475} \\
H^{21}(\bigwedge^{16} \cE^*) &=& H^{21}((\Gamma^{12,10,10}S)^{\oplus 3}) \\
		&=&(\Gamma^{(5,3,3,3,3,3,3,3,3,3)}V)^{\oplus 3} \simeq (\bC^{55})^3 \simeq \bC^{165}  \\
H^{21}(\bigwedge^{15} \cE^*) &=& H^{21}(\Gamma^{10,10,10}S) \\
		&=&\Gamma^{(3,3,3,3,3,3,3,3,3)}V \simeq \bC
\end{eqnarray*}

\begin{rema}
Standard properties of spectral sequences imply
$$-2816=\chi(\cO_{F_2(X)})= \sum (-1)^{p+q}
H^{q}(\bigwedge^{-p} \cE^*)= \sum (-1)^{a+p} H^{7a}(\bigwedge^{-p} \cE^*).$$
We compute 
$$-2816 =   1+6+(-55+306-435)+(-4950+2475-165+1) ,$$
thus everything checks.
\end{rema}

\begin{ques} The simplest structure on the spectral sequence would be to have exactness 
in all degrees where there is no contribution to cohomology, i.e., degeneration at $E_2$. Thus the only
nonzero terms in the second page would be
$$E^{-18,21}_2 \simeq \bC^{4950-2475+165-1} = \bC^{2639}, \quad
E^{-11,14}_2 \simeq \bC^{435-306+55}  = \bC^{184},$$
as well as
$$E^{-5,7}_2 \simeq \bC^6, \quad E^{0,0}_2\simeq \bC.$$
Does this degeneration occur? It would imply a nontrivial filtration on the holomorphic three-forms
of $F_2(X)$.  
\end{ques}

\begin{prop}
The Hodge diamond of $F_2(X)$ is
$$\begin{array}{ccccccc}
  &   &   &  1  &  &  &  \\
   &   &   0  &  & 0 & & \\
   &  6 &  & 62 & & 6 & \\
 2823  & &15684 & & 15684 & & 2823
 \end{array}  
$$
\end{prop}

\begin{proof}
We recall that 
\begin{itemize}
\item{The Picard rank $\rho(F_2(X))=1$
when $X$ is very general \cite[Th.~0.3]{Jiang}.}
\item{Let $\Sigma\ra \bP^2$ denote the double cover branched along degeneracy curve;
there exists an isometry on primitive cohomology
$$H^6(X,\bZ)_{\prim} \hookrightarrow H^2(\Sigma,\bZ)(-2)_{\prim},$$ 
realizing the former as an index-two sublattice of the latter \cite[Th.~0.1]{ogrady86}.}
\end{itemize}
The latter observation gives the Hodge numbers of $X$ displayed in
Section~\ref{subsect:recall}.

The incidence correspondence 
$$\begin{array}{rcccl}
	&	& Z & &  \\
	 &\stackrel{f}{\swarrow} & & \stackrel{g}{\searrow} & \\
X  &     &   & & F_2(X)
\end{array}
$$
induces Abel-Jacobi maps
$$
\alpha_1:g_*f^*:H^6(X,\bZ) \ra H^2(F_2(X),\bZ)(-2), 
$$
$$
\alpha_2:f_*g^*:H^4(F_2(X),\bZ) \ra H^6(X,\bZ)(1).
$$
Letting 
$$
L:H^2(F_2(X),\bZ) \ra H^4(F_2(X),\bZ)
$$ denote intersection by the
hyperplane class $\sigma_1$,
the composition
$$\alpha_2 \circ L \circ \alpha_1:H^6(X,\bZ) \ra H^6(X,\bZ)$$
forces the cohomology of $X$ and $F_2(X)$ to be tightly
intertwined. Such constructions are used to establish
Grothendieck's version of the Hodge conjecture for 
varieties with sparse Hodge diamond; see for example
\cite{VoisinConiveau}.

For our immediate purpose, we can apply the Main Theorem of
\cite{Shimada} to conclude that the cylinder map $\alpha_1$
injects the primitive cohomology $H^6(X,\bZ)_{\prim}$ into 
$H^2(F_2(X),\bZ)$. Since we computed 
$$h^{0,2}(F_2(X))=6$$
in Proposition~\ref{prop:getsix}, we conclude that 
$$
H^2(F_2(X),\bQ)_{\prim} \simeq H^6(X,\bQ)(2)_{\prim} \oplus \bQ(-1)^N
$$
for some $N$.
The last summand is of Hodge-Tate type. 
However, $N=0$ by Jiang's result on the Picard group.
\end{proof}

\bibliographystyle{alpha}
\bibliography{3quad}
\end{document}